\documentclass[11pt,twoside]{gsp}

\def\Sp{
{\rm Sp}}

\def\SO{
{\rm SO}}

\def\SU{
{\rm SU}}

\def\U{
{\rm U}}

\theoremstyle{plain}
\newtheorem{theorem}{Theorem}
\newtheorem{definition}[theorem]{Definition}

\newtheorem{proposition}[theorem]{Proposition}

\newtheorem{example}[theorem]{Example}

\newtheorem{remark}[theorem]{Remark}

\newenvironment{proof}[1][{}]{\par\noindent\textbf{Proof{#1}: }}{\hspace*{\fill}$\blacksquare$\smallskip\noindent\par}

\newcommand{\Z}{{\mathbb Z}}

\newcommand{\R}{{\mathbb R}}

\renewcommand{\H}{{\mathbb H}}
\renewcommand{\SO}{{\rm SO}}
\renewcommand{\U}{{\rm U}}

\renewcommand{\Sp}{{\rm Sp}}

\begin{document}

\thispagestyle{plain}

\title[Twistor spaces and (non)K\"ahler structures]{Twistor spaces and 
compact manifolds admitting both K\"ahler and non-K\"ahler structures}
\author[Ljudmila Kamenova]{Ljudmila Kamenova}

\date{}

\maketitle

\comm{Communicated by Ivailo Mladenov}\

\begin{abstract}
In this expository paper we review some twistor techniques and recall the 
problem of finding compact differentiable manifolds that can carry both 
K\"ahler and non-K\"ahler complex structures. Such examples were constructed 
independently by M. Atiyah, A. Blanchard and E. Calabi in the $1950$'s. 
In the $1980$'s V. Tsanov gave an example of a simply connected manifold 
that admits both K\"ahler and non-K\"ahler complex structures - the twistor 
space of a $K3$ surface. Here we show that the quaternion twistor space of 
a hyperk\"ahler manifold has the same property. This paper is dedicated to 
the memory of V. Tsanov. \\[0.2cm] 
\end{abstract}

\label{first}

\section[]{Introduction}

In this paper we discuss a couple of classical approaches to twistor theory. 
Roughly speaking, the twistor space $Z(M)$ is a family of (almost) complex 
structures on an orientable Riemannian manifold $(M, g)$ compatible with 
the given metric $g$ and the orientation. 
We are going to apply twistor techniques towards the problem of 
constructing simply connected compact manifolds that carry both 
K\"ahler and non-K\"ahler complex structures. 

Atiyah's idea behind his examples in \cite{aex} was to consider the set 
of all complex structures $M_n$ on the real torus $T^{2n} = \R^{2n}/\Z^{2n}$ 
coming from the complex vector space structures on $\R^{2n}$. The space 
$M_n$ is a complex manifold which is differentially a product of an 
algebraic variety and the torus $T^{2n}$, and therefore admits a K\"ahler 
structure. On the other hand, there exists a ``twisted'' complex structure 
on $M_n$ which is non-K\"ahler. This rationale works in many other cases, 
and in particular, one can produce simply connected examples of similar 
nature. In \cite{vv} Tsanov showed that the twistor space of a $K3$ 
surface is a simply connected $6$-dimensional compact manifold which 
carries both K\"ahler and non-K\"ahler complex structures. Here we give 
examples of twistor spaces of hyperk\"ahler manifolds and show that 
they also carry both K\"ahler and non-K\"ahler complex structures. 

Twistor theory was originally introduced by Penrose \cite{pen} and 
studied by Atiyah, Hitchin and Singer \cite{ahs} in four-dimensional 
Riemannian geometry. Given a four-dimensional manifold $M$ with a 
conformal structure, the twistor transform associates to $M$ the 
projective bundle of anti-self-dual spinors, which is a ${\mathbb S}^2$-bundle 
over $M$. In higher dimensions the notion of twistor space was generalized 
by B\'erard-Bergery and Ochiai in \cite{bbo}. Given a $2n$-dimensional 
oriented manifold with a conformal structure, the twistor space $Z(M)$ 
parametrizes almost complex structures on $M$ compatible with the 
orientation and the conformal structure. 
For quaternionic K\"ahler and hyperk\"ahler manifolds 
there is an alternative twistor generalization introduced by Salamon 
\cite{sal} and independently by B\'erard-Bergery (see Theorem 14.9 in 
\cite{bes}). 
For a quaternionic K\"ahler manifold $M$ (i.e., a manifold with 
holonomy group contained in $\Sp(n)\cdot \Sp(1)$) there is a natural 
${\mathbb S}^2$-bundle $Z_0(M)$ over $M$ of complex quaternionic structures 
on $M$, the ``quaternion twistor space''. The quaternion twistor 
construction applied to hyperk\"ahler manifolds gives examples of simply 
connected manifolds that carry both K\"ahler and 
non-K\"ahler complex structures. The idea behind this result is that on 
one hand, the quaternion twistor space of a hyperk\"ahler manifold is 
diffeomorphic to a product of K\"ahler manifolds, and on the other hand,  
the twistor complex structure is not K\"ahler. 

\begin{theorem}
Let $M$ be a hyperk\"ahler manifold of real dimension $4n$. Then 
the quaternion twistor space $Z_0(M)$ is a simply connected compact 
manifold that carries both K\"ahler and non-K\"ahler complex structures. 
\end{theorem}

\section{Basics on Twistor Geometry} 

Fix a scalar product and an orientation on the $2n$-dimensional 
real space $\R^{2n}$. Denote the type-DIII compact hermitian symmetric space 
$\frac{\SO(2n)}{\U(n)}$ by $\Gamma_n$. Notice that, as a hermitian symmetric 
space, $\Gamma_n$ is a complex K\"ahler manifold. The space $\Gamma_n$ can 
be identified with the set of complex structures on $\R^{2n}$, compatible 
with the metric and the given orientation. Indeed, if $J \in \SO(2n)$ with 
$J^2 = -1$ is a complex structure on $\R^{2n}$, then the group $\SO(2n)$ 
acts on $J$ by conjugation. The isotropy group is $\U(n) \subset \SO(2n)$. 
The real dimension of $\Gamma_n$ is $n(n-1)$. 

\begin{definition}
Let $M$ be an oriented Riemannian manifold of dimension $2n$. The holonomy 
group of $M$ is contained in $\SO(n)$. Denote by $P$ the associated 
principal $\SO(2n)$-bundle of orthonormal linear frames on $M$. The 
fiber bundle $Z(M) = P \times_{\SO(2n)} \Gamma_n \rightarrow M$ is called 
the {\em twistor bundle} or {\em twistor space} of $M$. Equivalently, 
$Z(M)$ can be identified with $P/\U(n)$. 
\end{definition}

The dimension of $Z(M)$ is even: indeed, 
${\rm dim}_\R Z(M) = {\rm dim}_\R M + {\rm dim}_\R \Gamma_n = 
2n + n(n-1) = n(n+1)$. 
The twistor space $Z(M)$ has a ``tautological'' almost complex structure $J$ 
which we describe here. 
Denote by $p: Z(M) \rightarrow M$ the projection. For a point $z \in Z(M)$, 
let its image be $x=p(z) \in M$. The fiber of $p$ over $x$ is 
$p^{-1}(x) = \Gamma_n$, and therefore $z \in \Gamma_n$ can be considered 
as a complex structure $I_z$ on the tangent space $T_xM$. The Riemannian 
connection of $M$ determines a splitting of the tangent space 
$T_z Z(M) = V_z \oplus H_z$ into vertical and horizontal parts. The 
vertical part $V_z$ is identified with the tangent space to the fiber 
$p^{-1}(x) = \Gamma_n$ at $z$, and it has an integrable almost complex 
structure $K$. The connection of $M$ defines an isomorphism 
$H_z \cong T_xM$ and we can consider $I_z$ as a complex structure on $H_z$. 

\begin{definition}
The almost complex structure $J$ on $Z(M)$ is defined by 
$J_z = K\oplus I_z: V_z \oplus H_z \rightarrow V_z \oplus H_z$. 
As $z$ varies in $p^{-1}(x) = \Gamma_n$, the projection of $J_z$ on 
the horizontal part $H_z \cong T_xM = \R^{2n}$ varies in the space of 
complex structures compatible with the metric and orientation. 
\end{definition}

When ${\rm dim}_\R (M) =4$, the fiber $\Gamma_2= {\mathbb S}^2$ and the 
generalized twistor definition coincides with the classical definition 
which was first introduced for Riemannian $4$-manifolds. 
Twistor spaces can also be defined in terms of spinors. In our notations 
$Z(M)$ coincides with the projectivized bundle ${\mathbb P} (P{\mathbb S}^+)$
of positive pure spinors on $M$ (Proposition 9.8 in \cite{lm}). 

\begin{proposition} 
Let $M$ be an oriented Riemannian manifold of dimension $2n$. Then the 
orthogonal orientation preserving almost complex structures on $M$ 
are in one-to-one correspondence with the sections of the twistor 
space $Z(M)={\mathbb P} (P{\mathbb S}^+)$. In particular, $M$ is K\"ahler 
if and only if there is a parallel cross section of $Z(M)$. 
\end{proposition}

From the definition of the almost complex structure $J$ on $Z(M)$ 
it is clear that the integrability of $J$ depends only on the Riemannian 
metric $g$ on $M$. When ${\rm dim}_\R (M) >4$, the Weyl tensor $W$ is 
irreducible, however when ${\rm dim}_\R (M) >4$, the Weyl tensor splits 
into two parts: self-dual and anti-self-dual $W=W_+ \oplus W_-$. 

\begin{theorem} \label{compl}
Let $M$ be an oriented Riemannian manifold of dimension $2n$. 
The almost complex structure $J$ on $Z(M)$ is integrable if and only if: 
\begin{itemize}
\item $M$ is anti-self-dual, i.e., the self-dual part of the Weyl tensor 
vanishes: $W_+=0$, when ${\rm dim}_\R (M) =4$ (Theorem 4.1 in \cite{ahs}); 
\item $M$ is conformally flat (i.e., $W=0$) when ${\rm dim}_\R (M) >4$ 
(section 3 in \cite{bbo}). 
\end{itemize}
\end{theorem}

Another important question in complex geometry is to determine when a 
given manifold is K\"ahler. 

\begin{theorem} \label{kahl}
The twistor space $Z(M)$ of a compact oriented Riemannian manifold $M$ 
is K\"ahler if and only if: 
\begin{itemize}
\item $M$ is conformally equivalent to the complex projective space 
${\mathbb {CP}}^2$ or to the sphere ${\mathbb S}^4$ when 
${\rm dim}_\R (M) =4$ (Theorem 6.1 in \cite{hit});
\item $M$ is conformally equivalent to the sphere ${\mathbb S}^{2n}$ when 
${\rm dim}_\R (M) =2n > 4$ (see \cite{bmn}). 
\end{itemize}
\end{theorem}


Notice that the sections of $p: Z(M) \rightarrow M$ represent almost 
complex structures on $M$ compatible with the Riemannian metric and the 
orientation. If the section is holomorphic, it represents an 
integrable almost complex structure. 

\begin{theorem} (Michelsohn, \cite{mic}) \label{mich}
Let $M$ be an oriented Riemannian manifold of even dimension with an 
almost complex structure $I$ determined by a projective spinor field 
$s \in \Gamma(Z(M))$. Then $I$ is integrable if and only if the section 
$s$ is holomorphic. 
\end{theorem} 

We conclude this section with some interesting examples of twistor spaces 
and fibers of the twistor projection in small dimensions. 

\begin{example}
Consider the sphere ${\mathbb S}^{2n}$ as the conformal compactification of 
$\R^{2n}$. Then from the topology of fiber bundles it is clear that 
$Z({\mathbb S}^{2n})= \Gamma_{n+1}$. This implies that there are no 
complex structures on ${\mathbb S}^{6}$ compatible with the standard metric. 
Indeed, if $s: {\mathbb S}^{6} \rightarrow Z({\mathbb S}^{6}) = \Gamma_4$ 
is a section representing such a complex structure, then 
$s({\mathbb S}^{6}) \subset \Gamma_4$ would be a complex submanifold 
by Theorem \ref{mich}. This would imply that $s({\mathbb S}^{6})$ is 
K\"ahler which is impossible since $H^2({\mathbb S}^{6}, \R)=0$ and 
there couldn't be any K\"ahler $(1,1)$ classes on ${\mathbb S}^{6}$. 
\end{example}

\begin{example}
Due to a lot of special isomorphisms between symmetric spaces of small 
dimensions, we have the following isomorphisms: 
\begin{itemize} 
\item
$\Gamma_1$ is a point, because $\SO(2) \cong \U(1)$; 
\item
$\Gamma_2 = {\mathbb S}^2 = {\mathbb {CP}}^1$, because 
$\frac{\SO(4)}{\U(2)} \cong \frac{\U(2)}{\U(1)}$; 
\item
$\Gamma_3 =  {\mathbb {CP}}^3$, because 
$\frac{\SO(6)}{\U(3)} \cong \frac{\U(4)}{\U(1) \times \U(3)}$; 
\item
$\Gamma_4 = Q_6$ is a complex $6$-dimensional quadric, because 
$\frac{\SO(8)}{\U(4)} \cong \frac{\SO(8)}{\SO(2) \times \SO(6)}$. 
\end{itemize}
\end{example}

\section{Examples of Twistor Spaces of Surfaces}

Let $M$ be a $4$-dimensional oriented Riemannian manifold. 
By the construction of the almost complex structure $J = K\oplus I$ on 
the twistor space $Z(M)$, the fibers of the twistor map 
$p: Z(M) \rightarrow M$ are holomorphic rational curves in $Z(M)$, 
called ``twistor lines''. The normal bundle of each twistor line 
is ${\cal O}(1) \oplus {\cal O}(1)$. From the results in \cite{ahs} we 
get the following universal property of twistor spaces. 

\begin{theorem} \label{jahs}
Let $M$ be an anti-self-dual $4$-manifold, i.e.,  $W_+=0$. Then the 
twistor space $Z(M)$ is a complex manifold that admits a ``real 
structure'', i.e., an anti-holomorphic involution 
$\iota: Z(M) \rightarrow Z(M)$. The restriction of the involution 
$\iota$ on each twistor line $\mathbb {CP}^1 = {\mathbb S}^2$ is the 
antipodal map. Conversely, the holomorphic data above is sufficient to 
define a twistor space. More precisely, if $Z$ is a $3$-dimensional 
complex manifold with an antiholomorphic involution $\iota$, foliated 
by rational curves with normal bundles ${\cal O}(1) \oplus {\cal O}(1)$, 
such that $\iota$ restricted to a fiber of the foliation $\mathbb {CP}^1$ is 
the antipodal map, 
then $Z$ is a twistor space of an anti-self-dual $4$-manifold. 
\end{theorem}

Here are the first examples where a twistor-type (almost) complex structure 
was explored even before twistor spaces were defined. 

\begin{example} 
Consider the $4$-dimensional real torus $M = T^4 = \R^4/\Z^4$. The twistor 
space $Z(T^4)$ of the torus is the quotient of 
${\cal O}(1) \oplus {\cal O}(1)$ by a corresponding lattice action. 
In \cite{blan} Blanchard considered this example and showed that 
$Z(T^4)$ is a compact complex manifold admitting a holomorphic 
fibration to $\mathbb {CP}^1$ whose fibers are complex tori. 
In \cite{cal} Calabi also explored $6$-dimensional examples 
of oriented Riemannian manifolds embedded into the Cayley space,  
which are diffeomorphic to a K\"ahler manifold but do not admit a 
K\"ahler metric. In \cite{aex} Atiyah considered fiber 
spaces of higher-dimensional complex tori that 
arise as twistor spaces. On one hand, Atiyah established the 
non-K\"ahler property of these fiber spaces, and on the other hand, 
he showed that they are diffeomorphic to the product of two complex 
K\"ahler manifolds (namely, the given complex torus and a symmetric 
space of type DIII). 
\end{example}

Atiyah was interested in these examples, because they illustrate the 
existence of K\"ahler and non-K\"ahler complex structures on the same 
differentiable manifold. In \cite{aex} he asked if a simply 
connected compact differentiable manifold can carry both K\"ahler 
and non-K\"ahler structures. The following example answers Atiyah's question. 

\begin{example} \label{VVexample}
In \cite{vv} Tsanov noticed that the twistor space $Z(S)$ of a $K3$ surface 
$S$ is simply connected and admits both K\"ahler and non-K\"ahler complex 
structures. Indeed, the quaternions $\H$ act as parallel endomorphisms 
on the tangent bundle $TS$. Fix the standard basis $\{I, J, K\}$ of $\H$. 
This gives a trivialization of $Z(S)$, i.e., $Z(S)$ is diffeomorphic 
to $S \times \mathbb {CP}^1$. Notice that since every $K3$ surface is 
K\"ahler, $S \times \mathbb {CP}^1$ admits a K\"ahler product structure, 
however $Z(S)$ endowed with the twistor complex structure is not 
K\"ahler by Hitchin's result (Theorem \ref{kahl}). 
\end{example}

For any anti-self-dual complex surface $S$ we can consider $Z(S)$ 
as the projectivized bundle ${\mathbb P} (P{\mathbb S}^+)$ 
of positive pure spinors on $S$ (Proposition 9.8 in \cite{lm}). For the 
tautological complex structure $J$ on $Z(S)$, the projection 
$p: (Z(S),J) \rightarrow S$ is not a holomorphic map in general. However, 
there exists a complex structure $J'$ on $Z(S)$ such that 
$p: (Z(S),J') \rightarrow S$ becomes a holomorphic $\mathbb {CP}^1$-bundle. 
Tsanov proved that $(Z(S), J)$ and $(Z(S), J')$ are not deformation equivalent 
to each other, and therefore the moduli space of complex structures on 
the complex manifold $Z(S)$ is not connected. The method used in \cite{vv} is 
an explicit computation of the first Chern classes of $(Z(S), J)$ and 
$(Z(S), J')$. Two complex structures on a $3$-dimensional complex manifold 
are homotopic if and only if their first Chern classes coincide. Tsanov 
shows that there is no diffeomorphism $\phi$ of $Z(S)$ such that 
$\phi^*$ sends $c_1(Z(S), J)$ to $c_1(Z(S), J')$. 

\begin{example} 
Let $\Sigma_g$ be a Riemannian surface of genus $g \geq 2$. 
Kato \cite{kato} considers the twistor space $Z(S)$ of the ruled surface 
$S=\mathbb {CP}^1 \times \Sigma_g$. By Theorem \ref{kahl}, $Z(S)$ is 
non-K\"ahler. Kato showed that there are no non-constant meromorphic 
functions on $Z(S)$. A classical result by Catanese states that the 
existence of a non-constant holomorphic map from a compact K\"ahler manifold 
to a compact Riemannian surface of genus $g \geq 2$ is determined by its 
topology. Kato's example shows that the K\"ahler assumption is essential and 
cannot be removed from Catanese's result. In \cite{lk} we described explicitly 
the complex structures of $S$ in terms of holomorphic sections 
$S \rightarrow Z(S)$. Kato's example fits into the theory of scalar-flat 
K\"ahler surfaces, which are explored in the key papers \cite{leb} of LeBrun 
and \cite{klp} of Kim, LeBrun and Pontecorvo.  

\end{example}

\section{Twistor Spaces of Hyperk\"ahler Manifolds}

\begin{definition}
A $4n$-dimensional Riemannian manifold is called hyperk\"ahler 
if its holonomy group is contained in the compact symplectic 
group $\Sp(n)$. 
\end{definition}

Since $\Sp(n) = \Sp(2n, {\mathbb C}) \cap \U(2n)$ is a subgroup of 
$\SU(2n)$, by Berger's 
classification, every hyperk\"ahler manifold is K\"ahler with zero 
Ricci curvature.

\begin{definition}
A $4n$-dimensional Riemannian manifold is called quaternion K\"ahler 
if its holonomy group is contained in $\Sp(n) \cdot \Sp(1)$. 
\end{definition}

On the other hand, $\Sp(n) \cdot \Sp(1)$ is not a subgroup of 
$\U(n)$, and therefore a general quaternion K\"ahler manifold is 
not K\"ahler. If $n \geq 2$, any quaternion K\"ahler manifold is 
Einstein. 

From now on we assume that $M$ is a hyperk\"ahler manifold of dimension 
${\rm dim}_\R M = 4n$. Then the imaginary quaternions act 
on $TM$ as parallel endomorphisms. Fix the standard basis $\{I, J, K\}$ of 
the imaginary quaternions. Both Salamon \cite{sal} and B\'erard-Bergery 
(see Theorem 14.9 
and Definition 14.67 in \cite{bes}) independently introduced an alternative 
notion of a twistor space for quaternion K\"ahler manifolds. Let $E$ be 
the $3$-dimensional vector subbundle of ${\rm End}(TM)$ spanned by 
$\{I, J, K\}$. The vector bundle $E$ carries a natural Euclidean structure, 
with respect to which $\{I, J, K\}$ is an orthonormal basis. 

\begin{definition}
The unit-sphere subbundle $Z_0(M)$ of $E$ is the quaternion 
twistor space of $M$.  
\end{definition}

Salamon \cite{sal} proved that the quaternion twistor space $Z_0(M)$ of $M$ 
admits a natural complex structure such that the fibers of the projection 
$\pi: Z_0(M) \rightarrow M$ are holomorphic rational curves called 
``twistor lines''. Notice that the complex structures $\{I, J, K\}$ give a 
trivialization of $Z_0(M)$, and therefore $Z_0(M)$ is diffeomorphic to 
$M \times {\mathbb S}^2 \cong M \times \mathbb {CP}^1$. 

We can also describe the quaternion twistor space $Z_0(M)$ of a hyperk\"ahler 
manifold $M$ as follows. We set $Z_0(M) = \mathbb {CP}^1 \times M$, where 
$\mathbb {CP}^1$ is identified with the unit sphere ${\mathbb S}^2 \subset 
\H$ of complex structures on $M$. We define the ``tautological'' (almost) 
complex structure $J'$ on $Z_0(M)$ in the same way as in the case of 
Riemannian $4$-folds. Denote by $\pi: Z_0(M) \rightarrow M$ the projection. 
Every point $z \in Z_0(M)$ is of the form $z = (\alpha, \pi(z))$, where 
$\alpha \in \mathbb {CP}^1$. Let $K$ be the natural complex structure of 
$\mathbb {CP}^1$. Then we define the ``tautological'' almost 
complex structure $J'$ on $T_zZ_0(M)$ by $J_z' = (K, \alpha)$, where 
$\alpha \in \mathbb {CP}^1$ is considered as a complex structure on 
$T_{\pi(z)} M$. This defines an integrable almost complex structure $J'$ on 
$Z_0(M)$ by Salamon \cite{sal} and $Z_0(M)$ becomes a $(2n+1)$-dimensional 
complex manifold. The quaternion twistor space $Z_0(M)$ is an almost 
complex submanifold (and a subbundle) of the ``big'' twistor space $Z(M)$. 

As Joyce notes in \cite{joy}, the quaternion twistor space $Z_0(M)$ 
of a hyperk\"ahler manifold $M$ 
satisfies very similar properties to the $4$-dimensional Riemannian 
case. The projection $p: (Z_0(M), J') \rightarrow \mathbb {CP}^1$ 
is holomorphic and $p^{-1}(\alpha)$ is isomorphic to the complex 
manifold $(M, \alpha)$ for any $\alpha \in \mathbb {CP}^1$. There is an 
antiholomorphic symplectic involution $\iota : Z_0(M) \rightarrow Z_0(M)$ 
defined by $\iota(\alpha, x) = (-\alpha, x)$. Consider the projection 
$\pi: Z_0(M) \rightarrow M$. For every point $x \in M$ the fiber over $x$ 
is called a {\it twistor line} and it is a holomorphic rational curve 
(\cite{sal}) with normal bundle ${\cal O}(1)^{\oplus 2n}$. The twistor lines 
are preserved by the involution $\iota$ and the restriction of $\iota$ on 
a twistor line coincides with the antipodal map. As in Theorem \ref{jahs}, 
the holomorphic data above is sufficient to show that a given 
$(2n+1)$-dimensional complex manifold with this data is biholomorphic to a 
quaternion twistor space of a hyper-complex manifold equipped with a 
pseudo-hyperk\"ahler metric (Theorem 7.1.4 in \cite{joy}).

\begin{theorem} \label{mainthm}
Let $M$ be a hyperk\"ahler manifold of real dimension $4n$. Then 
the quaternion twistor space $Z_0(M)$ is a simply connected compact 
manifold that carries both K\"ahler and non-K\"ahler complex structures. 
\end{theorem}

\begin{proof} 
From the definition of $Z_0(M)$ it is clear that the quaternion 
twistor space is diffeomorphic to the product of complex manifolds 
$M \times \mathbb {CP}^1$, which admits a product K\"ahler structure. 
We have that $\pi_1(Z_0(M)) = \pi_1(M) \times \pi_1({\mathbb S}^2)=0$, 
i.e., the quaternion twistor space is simply connected. 

Now consider $Z_0(M)$ together with its twistor complex structure $J'$. 
For a compact hyperk\"ahler manifold $M$ 
Fujiki \cite{fuj} showed that the general fiber of the twistor family 
$p: (Z_0(M), J') \rightarrow \mathbb {CP}^1$ does not contain neither 
effective divisors nor curves, and therefore the general fiber is not 
projective. 

Notice that $H^0(Z_0(M), \Omega^i_{Z_0(M)})=0$ for $i>0$. Let's first 
show this for $i=1$. If we assume there is a section 
$\sigma \in H^0(Z_0(M), \Omega^1_{Z_0(M)})$, then $\sigma$ defines a 
linear map $TZ_0(M)|_{\mathbb CP^1} \rightarrow \mathbb C$, where 
$\mathbb CP^1$ is a twistor line of the projection 
$\pi: Z_0(M) \rightarrow M$. Since the normal bundle of a twistor line 
is ${\cal O}(1)^{\oplus 2n}$ by \cite{sal}, we have the following 
splitting: $TZ_0(M)|_{\mathbb CP^1} = T{\mathbb CP^1} \oplus 
N_{{\mathbb CP^1}/Z_0(M)} = {\cal O}(2) \oplus {\cal O}(1)^{\oplus 2n}$ 
for every twistor line, i.e., for every fiber of $\pi$. 
But since the dual ${\cal O}(-1)$ of the hyperplane bundle 
doesn't have non-trivial sections, it follows that $\sigma=0$. Since 
we can express $\Lambda^i TZ_0(M)|_{\mathbb CP^1}$ as a direct sum of tensor 
powers of ${\cal O}(1)$, we can use the same argument for $i>1$. 

Assume that $Z_0(M)$ is K\"ahler, then $H^0(Z_0(M), \Omega^2_{Z_0(M)})=0$ and 
$H^2(Z_0(M))=H^{1,1}(Z_0(M)) \not=0$ 
and $H^{1,1}(Z_0(M))$ would contain rational classes, hence $Z_0(M)$ is 
projective by the Kodaira embedding theorem. However, every closed 
analytic subvariety of a projective variety is projective. Then the 
fibers of the twistor family $p: Z_0(M) \rightarrow \mathbb {CP}^1$ 
would be smooth projective varieties, which contradicts Fujiki's 
result. Therefore, the quaternion twistor space $Z_0(M)$ together with the 
tautological twistor complex structure is not K\"ahler, and the underlying 
simply connected differentiable manifold $M \times \mathbb {CP}^1$ carries 
both K\"ahler and non-K\"ahler complex structures. 
\end{proof}

\begin{remark}
Claude LeBrun pointed out an alternative argument that the complex 
structure $J'$ on the ``small'' twistor space $Z_0(M)$ is non-K\"ahler 
for a hyperk\"ahler manifold $M$ of complex dimension $2n$. Let 
$p: Z_0(M) \rightarrow \mathbb {CP}^1$ be the holomorphic twistor map. 
Hitchin's argument \cite{hit} on the action of the real structure on 
$H^{1,1}(Z_0(M))$ shows that if $J'$ is a K\"ahler sructure, then 
the anti-canonical line bundle $K^*$ is ample. In the hyperk\"ahler 
case, $K^* = p^* {\cal O} (2n+2)$ is a pull-back from the base $\mathbb {CP}^1$ 
via the projection $p: (Z_0(M), J') \rightarrow \mathbb {CP}^1$, 
and hence it is trivial on every fiber of 
$p: Z_0(M) \rightarrow \mathbb {CP}^1$. Therefore, its sections 
are pull-backs, too. Then the linear system of any power of $K^*$ 
collapses every fiber of $Z_0(M) \rightarrow \mathbb {CP}^1$. 
This contradicts ampleness of $K^*$. By the same argument, the algebraic 
dimension of $Z_0(M)$ is $a(Z_0(M))=1$ for a 
hyperk\"ahler manifold $M$. 
\end{remark}

If $M$ is a hyperk\"ahler manifold, then there exists a universal 
deformation space ${\cal U} \rightarrow \text{Def}(M)$ of $M$, where 
the base $\text{Def}(M)$ is smooth. The quaternion twistor space 
$p: Z_0(M) \rightarrow \mathbb {CP}^1$ induces a non-trivial map 
$\mathbb {CP}^1 \rightarrow \text{Def}(M)$. In \cite{fuj2} Fujiki showed 
that the corresponding 1-dimensional subspace of the Zariski tangent space 
$H^1(M, T^{1,0}M) \cong H^{1,1}(M)$ of $\text{Def}(M)$ is spanned by 
the K\"ahler class of the given hyperk\"ahler metric on $M$, which 
coincides with the Kodaira-Spencer class. 

\begin{example}
In \cite{le} Claude LeBrun gives the following example. Let $M$ be a 
hyperk\"ahler manifold of complex dimension $2n$ and  
let $f: \rightarrow \mathbb {CP}^1$ be a ramified cover 
of degree $k$. Consider the pull-back $\hat Z_0(M) = f^*Z_0(M)$ of the map 
$p: Z_0(M) \rightarrow \mathbb {CP}^1$ via $f$ and let ${\hat p}=f^*p$ be 
the associated holomorphic submersion. Then $\hat Z_0(M)$ is 
a complex $(2n+1)$-manifold with canonical line bundle 
$K={\hat p}^* {\cal O} (-2nk-2)$. LeBrun noticed that once we take 
a ramified cover of the base $\mathbb {CP}^1$ of the twistor space, then 
the deformation space of $\hat Z_0(M)$ is the deformation space of 
the corresponding rational curve in the Teichm\"uller space, which grows 
as the degree $k$ of the cover grows. The spaces of type $\hat Z_0(M)$ can 
also be constructed if one takes any rational curve in the Teichm\"uller 
space and restricts the universal fibration $\cal U$ to this curve. 
The methods used in the proof of our Theorem \ref{mainthm} also show that 
$\hat Z_0(M)$ is non-K\"ahler. LeBrun's example is related to his earlier 
work \cite{lebr}, where he shows that if $S$ is any complex surface with 
even Betti number $b_1$, and if $M= S \# k (-\mathbb {CP}^2)$ is its 
blow-up at a large number $k>>0$ of points, then the twistor space $Z(M)$ 
admits both K\"ahler and non-K\"ahler complex structures, and moreover, 
they have different Chern numbers. 
\end{example}

\begin{remark} 
Let $M$ be a hyperk\"ahler manifold of real dimension $4n$. When $n>1$, 
the twistor space $Z(M)$ is integrable if and only if $M$ is flat, i.e., 
$M$ is a torus, because the vanishing of the Weyl tensor (as required by 
Theorem \ref{compl}) for a Ricci-flat manifold implies the vanishing of the 
entire curvature tensor. However, it is still an interesting question to 
answer when the almost complex manifold $Z(M)$ is diffeomorphic to a product 
$M \times \Gamma_{2n}$.  In particular, it would be 
interesting to know if there is a relation that the 
characteristic classes of a hyperk\"ahler complex $2n$-fold $M$ have to 
satisfy so that $Z(M)$ is a trivial bundle. 
\end{remark}

\section*{Acknowledgements}
The author is grateful to Vasil V. Tsanov for introducing her to the 
subject of twistor spaces and for his support and mentoring during her 
undergraduate years at Sofia University in Bulgaria. He will be 
truly missed. The author also thanks Nigel Hitchin, Claude LeBrun and 
Simon Salamon for very interesting 
ideas and communications about twistor techniques. 
She is supported by a grant from the Simons Foundation/SFARI (522730, LK).


Ljudmila Kamenova \\
Department of Mathematics, Room 3-115 \\
Stony Brook University \\
Stony Brook, NY 11794-3651, USA \\
{\it E-mail address}: {\tt kamenova@math.stonybrook.edu}\\[0.3cm]

\label{last}
\end{document}